\documentclass[12pt, leqno]{amsart}  

\usepackage{graphicx}
\usepackage{amssymb,amsmath,amsthm,amscd,accents}
\usepackage{stackengine,scalerel}
\usepackage{mathrsfs}
\usepackage{lscape}
\usepackage{enumerate}
\usepackage{upgreek}
\usepackage{stmaryrd}
\usepackage[usenames,dvipsnames]{color}
\usepackage[colorlinks=true, pdfstartview=FitV,
 linkcolor=blue,citecolor=blue,urlcolor=blue]{hyperref}
\usepackage{tikz}
\tikzset{dynkdot/.style={circle,draw,scale=.38}}
\usetikzlibrary{arrows.meta,arrows}
\usepackage{bm}
\usepackage{marginnote}
\usepackage{verbatim}
\usepackage[normalem]{ulem}  
\usepackage{xspace}

\usepackage[all]{xy}

\allowdisplaybreaks[3]

\newcommand{\arxiv}[1]{\href{http://arxiv.org/abs/#1}{\texttt{arXiv:#1}}}

\newcommand{\nc}{\newcommand}
\numberwithin{equation}{section}

\newenvironment{blue}{\relax\color{blue}}{\hspace*{.5ex}\relax}
\newenvironment{purple}{\relax\color{blue}}{\hspace*{.5ex}\relax}

\newenvironment{magenta}{\relax\color{magenta}}{\hspace*{.5ex}\relax}
\newenvironment{jaune}{\relax\marginnote{\ber \scalebox{.6}{\sc{To be deleted}} \er}
  \color{Orchid}}{\hspace*{.5ex}\relax}

\newcommand{\bep}{\begin{purple}}
\newcommand{\eep}{\end{purple}}

\newcommand{\bj}{\begin{jaune}}
\newcommand{\ej}{\end{jaune}}
\newcommand{\beb}{\begin{blue}}
\newcommand{\eb}{\end{blue}}
\newcommand{\bgr}{\begin{jaune}}
\newcommand{\egr}{\end{jaune}}
\newcommand{\bem}{\begin{magenta}}
\newcommand{\eem}{\end{magenta}}

\nc{\eq}{\begin{myequation}} \nc{\eneq}{\end{myequation}}
\nc{\eqn}{\begin{myequationn}} \nc{\eneqn}{\end{myequationn}}

\newcommand{\berm}[1]{\begin{red}{}\marginnote{\fbox{\scshape\lowercase{M}}}%
#1}  

\newcommand{\bero}[1]{\begin{red}{}\marginnote{\fbox{\scshape\lowercase{O}}}%
#1}

\newcommand{\berE}[1]{\begin{red}{}\marginnote{\fbox{\scshape\lowercase{E}}}%
#1}

\newcommand{\berMH}[1]{\begin{red}{}\marginnote{\fbox{\scshape\lowercase{MH}}}%
#1}  

\nc{\cmtm}[1]{\color{Bittersweet}\marginnote{\fbox{M}}{#1}}

\nc{\hs}{\hspace*}
\nc{\ms}{\mspace}
\nc{\st}[1]{\{{#1}\}}

\nc{\qR}[1]{\ttq_{\mspace{-2mu}\raisebox{-.8ex}{${\scriptstyle{#1}}$}}}

\theoremstyle{plain}
\newtheorem{lemma}{Lemma}[section]
\newtheorem{proposition}[lemma]{Proposition}
\newtheorem{theorem}[lemma]{Theorem}

\newtheorem{corollary}[lemma]{Corollary}
\newtheorem{conjecture}[lemma]{Conjecture}

\theoremstyle{definition}
\newtheorem{remark}[lemma]{Remark}

\newtheorem{definition}[lemma]{Definition}

\nc{\Def}{\begin{definition}}
\nc{\edf}{\end{definition}}
\nc{\mTh}{\begin{mtheorem}}
\nc{\enmth}{\end{mtheorem}}
\nc{\Th}{\begin{theorem}}
\nc{\enth}{\end{theorem}}
\nc{\Prop}{\begin{proposition}}
\nc{\enprop}{\end{proposition}}
\nc{\Lemma}{\begin{lemma}}
\nc{\enlemma}{\end{lemma}}
\nc{\Cor}{\begin{corollary}}
\nc{\encor}{\end{corollary}}
\nc{\Rem}{\begin{remark}}
\nc{\enrem}{\end{remark}}
\nc{\Conj}{\begin{conjecture}}
\nc{\enconj}{\end{conjecture}}

\nc{\ledot}{\mathrel{\le\ms{-11mu}\raisebox{.2ex}{$\cdot$}}}
\nc{\predot}{\mathrel{\preceq\ms{-9mu}\raisebox{.35ex}{$\centerdot$}}}

\renewcommand{\le}{\leqslant}
\renewcommand{\ge}{\geqslant}
\renewcommand{\preceq}{\preccurlyeq}

\newcommand{\seteq}{\mathbin{:=}}

\newcommand{\soplus}{\mathop{\mbox{\normalsize$\bigoplus$}}\limits}

\newcommand{\tens}{\mathop\otimes}

\newcommand{\g}{\mathfrak{g}}

\newcommand{\n}{\mathfrak{n}}

\newcommand{\Q}{\mathbb{Q}}
\newcommand{\Z}{\mathbb{Z}\ms{1mu}}

\newcommand{\al}{{\ms{1mu}\alpha}}
\newcommand{\ep}{\epsilon}
\newcommand{\la}{\lambda}
\newcommand{\be}{{\ms{1mu}\beta}}

\newcommand{\La}{\Lambda}

\newcommand{\wt}{{\rm wt}}
\newcommand{\ch}{{\rm ch}}

\newcommand{\im}{\imath}
\newcommand{\jm}{\jmath}

\newcommand{\ocalD}{\overline{\calD}}

\newcommand{\uw}{\underline{w}}

\newcommand{\sfC}{\mathsf{C}}
\newcommand{\sfD}{\mathsf{D}}

\newcommand{\sfP}{\mathsf{P}}
\newcommand{\sfQ}{\mathsf{Q}}

\newcommand{\sfS}{\mathsf{S}}

\newcommand{\sfW}{\mathsf{W}}

\newcommand{\sfc}{\mathsf{c}}

\newcommand{\bbD}{\mathbb{D}}

\newcommand{\bfM}{\mathbf{M}}

\newcommand{\calA}{\mathcal{A}}

\newcommand{\calD}{\mathcal{D}}
\newcommand{\calE}{\mathcal{E}}
\newcommand{\calF}{\mathcal{F}}

\newcommand{\calQ}{\mathcal{Q}}

\newcommand{\calU}{\mathcal{U}}

\newcommand{\hcalA}{\widehat{\calA}}

\newcommand{\scrC}{\mathscr{C}}

\newcommand{\ttB}{\mathtt{B}}

\newcommand{\ttb}{\mathtt{b}}

\newcommand{\ttq}{\mathtt{q}}

\nc{\bg}{\sigma}

\newcommand{\ttx}{\mathtt{x}}
\newcommand{\tty}{\mathtt{y}}
\newcommand{\ttz}{\mathtt{z}}

\newcommand{\rmE}{\mathrm{E}}

\newlength{\mylength}
\newcommand*{\para}{%
  \rlap{\rotatebox{-30}{\rule[.05ex]{.4pt}{.77em}}}%
  \kern.04em%
  \rlap{\kern.36em\raisebox{0.649519052835em}{\rule{.6em}{.4pt}}}%
  \rule{.6em}{.4pt}\kern-.04em%
  \rotatebox{-30}{\rule[.05ex]{.4pt}{.77em}}}

\newcommand{\isoto}[1][]{\mathop{\xrightarrow%
[{\raisebox{.3ex}[0ex][.3ex]{$\scriptstyle{#1}$}}]%
{{\raisebox{-.6ex}[0ex][-.6ex]{$\mspace{2mu}\sim\mspace{2mu}$}}}}}

\newcommand{\lan}{\langle}
\newcommand{\ran}{\rangle}

\newcommand{\qt}[1]{\quad\text{#1}}
\newcommand{\qtq}[1][{and}]{\quad\text{{#1}}\quad}

\newcommand{\ee}{\end{enumerate}}
\newcommand{\bitem}{\begin{itemize}}
\newcommand{\eitem}{\end{itemize}}
\newcommand{\ben}{\begin{enumerate}[{\rm (1)}]}
\newcommand{\bnum}{\begin{enumerate}[{\rm (i)}]}
\newcommand{\bnump}{\begin{enumerate}[{\rm (i)$'$}]}
\newcommand{\bna}{\begin{enumerate}[{\rm (a)}]}
\newcommand{\bnA}{\begin{enumerate}[{\rm (A)}]}
\newcommand{\bc}{\begin{cases}}
\newcommand{\ec}{\end{cases}}
\newcommand{\ba}{\begin{array}}
\newcommand{\ea}{\end{array}}

\newcommand{\snoi}{\smallskip \noindent}

\newenvironment{myequation}
{\relax\setlength{\arraycolsep}{1pt}\begin{eqnarray}}
{\end{eqnarray}}
\newenvironment{myequationn}
{\relax\setlength{\arraycolsep}{1pt}\begin{eqnarray*}}
{\end{eqnarray*}}

\nc{\eqs}[1]{\underset{\raisebox{.4ex}[.7ex][0ex]{$\scriptstyle{#1}$}}{=}}

\newcommand{\Dcan}{{\bbD_\can}}

\newcommand{\Ang}[1]{  \bigl\lan #1 \bigr\ran  }

\newcommand{\Es}{\rmE^\star}

\newcommand{\qintc}[3]{  \left[ \begin{matrix} #1 \\ #2 \end{matrix} \right]_{#3}  }

\newcommand{\longepito}[1][]{\xymatrix@C=4ex{{}\ar@{->>}[r]^{#1}&{}}}

\newcommand{\hAform}[1]{\bl #1\br_{\hcalA}\ms{1mu}}

\newcommand{\bl}{\bigl(}
\newcommand{\br}{\bigr)}

\newcommand{\oprod}{\prod^{\xrightarrow{}}}

\newcommand{\uii}{{\underline{\boldsymbol{\im}}}}

\nc{\ake}[1][2ex]{\rule[-.5ex]{0ex}{#1}}

\newcommand{\ujj}{{\underline{\boldsymbol{\jm}}}}
\newcommand{\TT}{ \textbf{\textit{T}}}

\nc{\col}{\colon}
\nc{\ord}{\mathrm{ord}}
\nc{\catQ}{\scrC_\calQ}
\nc{\catD}{\scrC_\bbD}
\nc{\vs}{\vspace*}
\nc{\D}{\mathscr{D}}
\nc{\Proof}{\begin{proof}}
  \nc{\QED}{\end{proof}}

\nc{\tLa}{\widetilde{\Lambda}}
\nc{\ro}{{\rm(}}
\nc{\rf}{{\rm)}\xspace}
\nc{\Aut}{\mathrm{Aut}}
\nc{\can}{\mathrm{can}}
\nc{\Dc}{{\bbD_{\can}}}
\nc{\Cgo}{\scrC_{\mathfrak{g}}^{0}}
\nc{\bb}{\mathtt{b}}
\nc{\catC}{\scrC}
\nc{\braid}{\ttB}
\nc{\Ass}[1][\bbD]{\calE_{#1}}
\nc{\Ci}{C^\uii}
\nc{\Cj}{C^\ujj}
\nc{\condi}[1][K]{with $#1\cap\st{0,1}\not=\emptyset$\xspace}
\nc{\Di}[1][{\uii}]{{\bbD,{#1}}}
\nc{\Vi}[1][\uii]{V^{#1}}
\nc{\Pii}[1][\uii]{P^{#1}}
\nc{\Vdi}[1][{\Di}]{V^{#1}}
\nc{\Pdi}[1][{\Di}]{P^{#1}} 
\nc{\PBix}[1][{[l,r]}]{\Z_{\ge0}^{\oplus{#1}}}
\nc{\mcf}[1][{$[a,b]$}]{maximal commuting family of $i$-boxes in {#1}\xspace}
\nc{\nn}{\nonumber}
\nc{\bwr}{\mbox{\large$\wr$}}
\newcommand{\tch}{\widetilde{\ch}}
\nc{\tchDcan}{\ms{2mu}\tch_{\mspace{.1mu}\raisebox{-.4ex}{${\scriptstyle{\Dcan}}$}}}
\nc{\monoto}[1][]{\xymatrix@C=2ex{\ar@{>->}[r]^-{{#1}}&}\ms{-8mu}}

\usepackage[colorinlistoftodos]{todonotes}

\title[Faithful action on bosonic extensions]{Faithful action of braid group \\ on bosonic extensions}

\author[M. Kashiwara]{Masaki Kashiwara}
\thanks{The research of M.\ Kashiwara
	was supported by Grant-in-Aid for Scientific Research (B)  23K20206,  
	Japan Society for the Promotion of Science.}
\address[M. Kashiwara]{%
Kyoto University Institute for Advanced Study, Research Institute
for Mathematical Sciences, Kyoto University, Kyoto 606-8502, Japan
}
\email[M. Kashiwara]{masaki@kurims.kyoto-u.ac.jp}

\author[M. Kim]{Myungho Kim}
\address[M. Kim]{Department of Mathematics, Kyung Hee University, Seoul 02447, Korea}
\email[M. Kim]{mkim@khu.ac.kr}
\thanks{The research of M.\ Kim was supported by the National Research Foundation of
Korea (NRF) Grant funded by the Korea government(MSIT)
(NRF-2020R1A5A1016126).}

\author[S.-j. Oh]{Se-jin Oh}
\thanks{ The research of S.-j.\ Oh was supported by the National Research Foundation of
	Korea (NRF) Grant funded by the Korea government(MSIT) (NRF-2022R1A2C1004045).}
\address[S.-j. Oh]{ Department of Mathematics, Sungkyunkwan University, Suwon, South Korea}
\email[S.-j. Oh]{sejin092@gmail.com}

\author[E. Park]{Euiyong Park}
\thanks{The research of E.\ Park was supported by the National Research Foundation of Korea (NRF) Grant funded by the Korea Government(MSIT)(RS-2023-00273425 and NRF-2020R1A5A1016126).}
\address[E. Park]{Department of Mathematics, University of Seoul, Seoul 02504, Korea}
\email[E. Park]{epark@uos.ac.kr}

\subjclass[2010]{17B37, 20F36}

\date{December 3, 2025}

\begin{document}

\begin{abstract} 
The braid group action on the bosonic extension $\hcalA$ of the quantum group $\calU_q(\g)$ has been introduced in recent works, and it can 
be regarded as a generalization of Lusztig's symmetries on $\calU_q(\mathfrak{g})$.
In this notes, we prove the faithfulness of this braid group action.
\end{abstract}

\maketitle
\tableofcontents

\section{Introduction}

The braid group symmetries, introduced by Lusztig~\cite{LusztigBook} (see also~\cite{Saito94}), provide a quantum analogue of the classical Weyl group $\sfW$ actions $\{ s_i \}_{i \in I}$ on Lie algebras, establishing a connection between root system symmetries and the structure of the quantum group $\calU_q(\g)$. More precisely, the automorphisms $\{\sfS_i\}_{i \in I}$ (see~\eqref{eq: Lusztig action}) satisfy the braid relations and send each homogeneous element $x \in \calU_q(\g)$ of weight $\beta$ to $\sfS_i(x)$ of weight $s_i(\beta)$.

Using this braid symmetry, one can construct, for each $w  \in \sfW$, the quantum coordinate subalgebra $A_q(\n(w))$ of $A_q(\n) \simeq \calU^-_q(\g)$ and its dual PBW basis $\sfP_{\uw}$ for each reduced expression $\uw$ of $w$. These subalgebras play an important role in various areas of mathematics. Nevertheless, it remains unknown whether the braid group action on $\calU_q(\g)$ via $\{\sfS_i\}_{i \in I}$ is faithful. %

The bosonic extension $\hcalA$, introduced by Hernandez and Leclerc for simply-laced finite types, is a $\Q(q^{1/2})$-algebra  generated by a $\mathbb{Z}$-indexed family of Chevalley generators $\{ f_{i,m} \}_{i\in I, m\in \mathbb{Z}}$ satisfying $q$-Serre, $q$-boson, and distant $q$-commutation relations (see~\eqref{eq: bosonic extension}). Each subalgebra $\hcalA[m]$ generated by $\{ f_{i,m} \}_{i\in I}$ is isomorphic to $U_q^-(\g)$, prompting the question of whether $\hcalA$ admits a braid symmetry.
This was confirmed in~\cite{KKOP21B,JLO2} for finite types and later extended to arbitrary symmetrizable types in~\cite{KKOP24B}
by constructing automorphisms $\{\TT_i\}_{i \in I}$ satisfying the braid relations.
Using this symmetry, the subalgebra
$\hcalA(\ttb) \seteq \TT_\ttb \hcalA_{<0} \cap \hcalA_{\ge 0}$
is defined for each element $\ttb$ in the braid monoid $\ttB^+$ of the braid group $\ttB$,
and PBW-type bases $\sfP_{\underline{\ttb}}$ of $\hcalA(\ttb)$ associated with every expression $\underline{\ttb}$ of $\ttb$ were established (see Lemma~\ref{lem:int}).

In this paper, we prove that the braid group action on $\hcalA$ via $\{\TT_i\}_{i \in I}$ is faithful in the finite type case (Theorem~\ref{thm: main}). It is well known that finite-type braid groups possess a distinguished element $\Updelta \in \ttB$, corresponding to the longest element of the Weyl group $\sfW$, whose powers generate central elements.
Every element $\ttb \in \ttB$ can be expressed as
a product of a power of $\Delta$ and prefixes of $\Updelta$,
known as the Garside normal form. In addition to these properties, we exploit the non-degenerate symmetric bilinear form $\hAform{\ ,\ }$, the adjoint pair of endomorphisms $(f_{i,m}, \rmE_{i,m})$ and PBW-bases to study the action of the braid symmetry on $\hcalA$. Using these tools, we establish the faithfulness of the braid group action, which provide a foundation for further study of the interplay between braid group and the algebraic structure of the bosonic extension.

\section{Bosonic extension}

In this preliminary section, we recall the bosonic extension of quantum groups. 
The bosonic extension $\hcalA$ (\cite{KKOP24}) can be regarded as an affinization of a half of a quantum group, and it is 
isomorphic to the quantum Grothendieck ring of Hernandez-Leclerc category over a quantum affine algebra 
of untwisted types provided that $\hcalA$ is of simply-laced finite type (\cite{HL15} see also \cite{JLO2}). Throughout this paper, 
we restrict our attention to bosonic extensions of finite type.

\medskip

\subsection{Cartan matrix and associated data}
Let $\sfC = (\sfc_{i,j})_{i,j\in I}$ be a Cartan matrix of finite type,
$\Pi = \{ \al_i \}_{i \in I}$ the set of its corresponding simple roots,
and $\Pi^\vee = \{ h_i \}_{i \in I}$ the set of simple coroots,
which satisfy $\Ang{h_i,\al_j}=\sfc_{i,j}$.  
Note that $\sfC$ is symmetrizable in the sense that there exists a diagonal matrix $\sfD = {\rm diag} (d_i \in \Z_{\ge 1})$ 
such that  $\min(d_i)_{i \in I}=1$ and $\sfD\sfC$ is symmetric. We denote by $\sfP =\bigoplus_{i \in I} \Z \La_i$ the weight lattice and by $\sfQ =\bigoplus_{i \in I} \Z \al_i$
the root lattice corresponding to $\sfC$. Here $\La_i$ represents the $i$-th fundamental weight; i.e, $\Ang{h_j,\La_i} =\delta_{j,i}$. 

Note that there exists a $\Q$-valued symmetric bilinear form $( \ , \ )$ on $\sfP$ such that $(\al_i,\al_i)=2 d_i$ and $\Ang{h_i,\la} = 2(\al_i,\la)/(\al_i,\al_i)$ for any $i\in I$ and $\la \in \sfP$. 

\subsection{Bosonic extensions}
Let $q$ be an indeterminate with the formal square root $q^{1/2}$. For each $i \in I$, we set $q_i \seteq q^{d_i}$,
$$ [n]_i \seteq \dfrac{q_i^n-q_i^{-n}}{q_i-q_i^{-1}}, \ \ [n]_i! \seteq \prod_{k=1}^n [k]_i \qtq \qintc{m}{n}{i} = \dfrac{[m]_i!}{[n]_i![m-n]_i!}$$
for $i \in I$ and $m \ge n \in \Z_{\ge0}$. 

The bosonic extension $\hcalA$ of the quantum group associated with $\sfC$ is defined as the $\Q(q^{1/2})$-algebra 
generated by the infinite family of generators $\{ f_{i,p} \}_{(i,p) \in I \times \Z }$ subject to the following relations:
\begin{equation}\label{eq: bosonic extension}
\parbox{75ex}{
\bna
\item $\displaystyle\sum_{k=0}^{1-\sfc_{i,j}} (-1)^k \qintc{1 - \sfc_{i,j}}{k}{i} f_{i,p}^kf_{j,p}f_{i,p}^{1-\sfc_{i,j}-k} =0$ for any $i \ne j \in I$ and $p \in \Z$, 
\item   $f_{i,m}f_{j,p} = q_i^{(-1)^{m+p+1} \sfc_{i,j}} f_{j,p}f_{i,m}  + \delta_{(j,p),(i,m+1)}(1-q_i^2)$ if $m <p$. 
\ee
  }
\end{equation}

With the assignment $\wt(f_{i,m})=(-1)^{m+1}\al_i$, the relations of $\hcalA$ in~\eqref{eq: bosonic extension} are homogeneous, and hence $\hcalA$
admits a $\sfQ$-weight space decomposition: 
\begin{align} \label{eq: weight decomp A}
\hcalA = \bigoplus_{\be \in \sfQ} \hcalA_\be. 
\end{align} 
We say that an element $x \in \hcalA_\be$ is homogeneous of weight $\be$ and set $\wt(x) \seteq \be$.

Note that $\hcalA$ has (i) the $\Q(q^{1/2})$-algebra automorphism $\ocalD$ defined by $\ocalD(f_{i,p})=f_{i,p+1}$ and  
(ii) the $\Q(q^{1/2})$-algebra anti-automorphism $\star$ defined by $(f_{i,p})^\star = f_{i,-p}$. 

\begin{definition}
For $-\infty \le a \le b \le \infty$, let $\hcalA[a,b]$ be the $\Q(q^{1/2})$-subalgebra of $\hcalA$ generated by $\{ f_{i,k} \ | \  i \in I, a \le k \le b \}$. We write
$$
\hcalA[m] \seteq \hcalA[m,m], \ \ \hcalA_{\ge m} \seteq \hcalA[m,\infty] \qtq \hcalA_{\le m} \seteq \hcalA[-\infty,m].
$$
Similarly $\hcalA_{> m} \seteq \hcalA_{\ge m+1}$ and $\hcalA_{< m} \seteq \hcalA_{\le m-1}$. 
\end{definition}

\begin{theorem} [{\cite[Corollary 5.4]{KKOP24}}]
  For any $m\in\Z$, the subalgebra $\hcalA[m]$ is isomorphic to the negative half $\calU_q^-(\g)$ of the quantum group $\calU_q(\g)$ associated with $\sfC$.
  Here $\calU_q(\g)$ \ro resp.\ $\calU_q^-(\g)$ \rf
  denotes the $\Q(q^{1/2})$-algebra generated by the Chevalley generators $e_i, f_i$ $(i \in I)$, and $t_i\seteq q_i^{h_i}$ \ro resp.\ $f_i$\rf.
Moreover, for any $a \le b$, the $\Q(q^{1/2})$-linear map
\begin{align} \label{eq: serial}
\hcalA[b] \otimes_{\Q(q)^{1/2}} \hcalA[b-1]  \otimes_{\Q(q)^{1/2}} \cdots \otimes_{\Q(q)^{1/2}} \hcalA[a+1]  \otimes_{\Q(q)^{1/2}}
\hcalA[a]  \to \hcalA[a,b]
\end{align}
defined by $x_b \otimes x_{b-1} \otimes \cdots \otimes x_{a+1} \otimes x_a \mapsto x_b x_{b-1} \cdots x_{a+1} x_a$ is an isomorphism. 
\end{theorem}

\subsection{Bilinear form and homomorphisms}
From~\eqref{eq: weight decomp A} and~\eqref{eq: serial},  $\hcalA$ admits the decomposition 
\begin{align} \label{eq: serial decom A}
  \hcalA =   \soplus_{(\be_k)_{k \in \Z} \in \sfQ^{\oplus Z}}
\oprod_{k \in \Z} \hcalA[k]_{\be_k},
\end{align}
where
$$\oprod_{k \in \Z} \hcalA[k]_{\be_k}=
\cdots \hcalA[1]_{\be_1}\hcalA[0]_{\be_0}\hcalA[-1]_{\be_{-1}}\cdots. $$

\begin{definition} [{\cite[\S 5, 6]{KKOP24}}]\hfill
\bna
\item Define $\bfM: \hcalA \longepito \Q(q^{1/2})$  to be the natural projection 
\begin{subequations} \label{eq: defn of forms on A}
\begin{align} \label{eq: M surjection}
\hcalA \longepito \oprod_{k \in \Z} \hcalA[k]_{0} \simeq \Q(q^{1/2}).    
\end{align}
\item Define a bilinear form $\hAform{ \ , \ }$ on $\hcalA$ as follows:
\begin{align} \label{eq: A form}
\hAform{x,y} \seteq \bfM(x \ocalD(y)) \in \Q(q)^{1/2} \quad \text{ for any } x,y \in \hcalA.
\end{align}
\item For homogeneous elements $x,y \in \hcalA$, we set 
\begin{align}
[x,y]_q \seteq xy - q^{-(\wt(x),\wt(y))} yx 
\end{align}
and extend this to non-homogeneous elements of $\hcalA$
via~\eqref{eq: weight decomp A}.
\item  For $(i,m) \in I \times \Z$, define endomorphisms $\rmE_{i,m}$ and $\Es_{i,m}$ of $\hcalA$ by
\begin{align} \label{eq: E endo}
\rmE_{i,m}(x) \seteq [x,f_{i,m+1}]_q \qtq  \Es_{i,m}(x) \seteq [f_{i,m-1},x]_q \ \  \text{ for } x \in \hcalA. 
\end{align}
\end{subequations}
\ee 
\end{definition}

Note that
\eq
[\hcalA_{<m},\hcalA_{>m}]_q=0\qt{for any $m\in\Z$.}\label{eq:br}
\eneq
 
\begin{theorem}[{\cite[\S 5]{KKOP24}}] \label{thm: hAform}
The bilinear form $\hAform{ \ , \ }$ is  non-degenerate and symmetric. Furthermore the form $\hAform{ \ , \ }$ satisfies the following properties:
\bna
\item $\hAform{x,y} = \hAform{\ocalD(x), \ocalD(y)}=\hAform{x^\star,y^\star}$ for any $x,y \in \hcalA$.
\item $\hAform{f_{i,m}x, y } =  \hAform{x, y f_{i,m+1} }$ and $\hAform{xf_{i,m}, y } =  \hAform{x, f_{i,m-1} y }$ for any $x,y \in \hcalA$.
\item \label{it: pairing} For any $x,y\in\hcalA_{\le m}$ and $u, v \in\hcalA_{\ge m}$, we have
$$
\hAform{f_{i,m}x,y}=\hAform{ x,\rmE_{i,m}(y)} \qtq \hAform{u , vf_{i,m} }=\hAform{\Es_{i,m}(u) ,v}.
$$
\item $\hAform{x,y}=0$ if $x,y$ are homogeneous elements such that
  $\wt(x)\not=\wt(y)$. 
\item For $x = \displaystyle\oprod_{k \in [a,b]} x_k\seteq x_bx_{b-1}\cdots x_a$ and $y = \displaystyle\oprod_{k \in [a,b]} y_k$ with $x_k,y_k \in \hcalA[k]$, we have
$$
\hAform{x,y} = q^{\sum_{s<t} (\wt(x_s),\wt(x_t))} \prod_{k\in [a,b]}
\hAform{x_k,y_k}. 
$$
\ee
\end{theorem}

\section{Braid group action on $\hcalA$}  

In this section, we first recall the braid group action on the quantum group $\calU_q(\g)$ following \cite{LusztigBook} (see also \cite{Saito94}).
We then review the braid group action on the bosonic extension $\hcalA$, introduced in \cite{KKOP21B,JLO2,KKOP24B}.
For a comparison between the braid group actions on $\calU_q(\g)$ and $\hcalA$, we refer the reader to the introduction of \cite{KKOP24B}.

\subsection{Braid and Weyl groups}
We denote by $\ttB_\sfC$ the braid group associated with $\sfC$; i.e, it is the group generated by $\{ \sigma_i\}_{i \in I}$
subject to the following relations:
\begin{align} \label{eq: braid relation}
  \underbrace{\sigma_{i}\sigma_{j}\cdots}_{m_{i,j}\text{-times}} =\underbrace{\sigma_{j}\sigma_{i}\cdots}_{m_{i,j}\text{-times}}
\qt{for $i \ne j \in I$,    }
\end{align}
where 
$m_{i,j} \seteq 2,3,4,6$ according to $c_{i,j}c_{j,i}=0,1,2,3$ respectively.     

We denote by $\ttB^\pm_\sfC$ the submonoid of $\ttB_\sfC$ generated by $\{ \sigma_i^\pm\}_{i \in I}$.

Note that there exists
a group automorphism
\eq
  \psi \col \ttB_\sfC \isoto \ttB_\sfC,
\eneq
which sends $\sigma_i$ to $\sigma_i^{-1}$ for all $i \in I$. 

\smallskip

Let $\sfW_\sfC$ denote the Weyl group associated with $\sfC$, generated by simple reflections $\{ s_i \}_{i \in I}$,
subject to the following relations:
$$\text{(i) $s_i^2= 1$}  \qtq \text{(ii) $\underbrace{s_{i}s_{j}\cdots}_{m_{i,j}\text{-times}} =\underbrace{s_{j}s_{i}\cdots}_{m_{i,j}\text{-times}}$ for $i \ne j \in I$}.$$
Note that $\sfW_\sfC$ contains the longest element $w_\circ$ and that $w_\circ$ induces an involution  $*:I \to I$ sending $i \mapsto i^*$
where $w_\circ(\al_i) = - \al_{i^*}$. 
We usually drop $_\sfC$ in the above notations if there is no danger of confusion.

\smallskip

We write $\pi\col \ttB \to \sfW$ the canonical group homomorphism sending $\sigma_i\mapsto s_i$.
We define $\Updelta$ to be the element in $\ttB^+$
such that $\ell(\Updelta)=\ell(w_\circ)$ and $\pi(\Updelta)=w_\circ$.
Here $\ell$ is the length function.  
We remark that $\Updelta^2$
is contained in the center of $\ttB$. 

\Lemma [{see \cite[Corollary 7.3]{OP24}}] \label{lem: loc red braid}
For any $\ttx \in \ttB$, there exist $\tty \in \ttB^+$ and $m \in \Z_{\ge 0}$ such that $\ttx\tty =\Updelta^m$.     
\enlemma

For $\ttx,\ttz \in \ttB$, we write $\ttx \ledot \ttz$ if there exists $\tty \in \ttB^+$ such that $\ttx\tty =\ttz$, or equivalently $\ttx^{-1}\ttz\in\ttB^+$. 

It is easy to see
\eq
\text{for any $\ttb_1$, $\ttb_2$,
 we have $\ttb_1\ledot \ttb_2\Longleftrightarrow \psi(\ttb_2)\ledot\psi(\ttb_1)$.}
\eneq

\begin{proposition} [{\cite{Gar69} and see also \cite[Chapter 6.6]{KT08}}]\label{prop: gcd} 
The partial ordered set $\ttB$ with the partial order $\ledot$ is a lattice;
i.e., every pair of elements of $\ttB$ has an infimum and a supremum. 
\end{proposition}

It is easy to see
\eqn
\text{for any $\ttb_1$, $\ttb_2$,
 we have $\ttb_1\ledot \ttb_2\Longleftrightarrow \psi(\ttb_2)\ledot\psi(\ttb_1)$.}
\eneqn
Hence we have
\eq
\psi(\ttb_1\wedge\ttb_2)=\psi(\ttb_1)\vee\psi(\ttb_2).
\label{eq:psibb}
\eneq

The infimum of $\ttx$ and $\ttz$ in $\ttB$ is denoted by $\ttx \wedge \ttz$
and the supremum is denoted by $\ttx \vee \ttz$.

\begin{theorem} [{Garside left normal form (see \cite{WP,EM94})}] \label{thm: Garside} 

Each element $\ttb \in \ttB$ can be presented
as 
$$
\Updelta^r \ttx_1 \cdots \ttx_k,
$$
where $r \in \Z$, $k \in \Z_{\ge0}$, $1 \lessdot \ttx_s \lessdot \Delta$
, and $\ttx_s=\Updelta\wedge(\ttx_s\ttx_{s+1})$ for $1 \le s <k$.
\end{theorem}
Note that the condition for the Garside normal form of $\ttb$
is that $r$ is the largest integer such that $\Updelta^{-r}\ttb\in\ttB^{+}$,
and $k$ is the largest integer such that $\ttx_k\not=1$, where
$\ttx_j\seteq\bl (\ttx_1\cdots \ttx_{j-1})^{-1}\Updelta^{-r}\ttb\br\wedge\Updelta$ for any
$j\in\Z_{>0}$.

\subsection{Braid group actions on $\calU_q(\g)$ and $\hcalA$}

It is well known that there exists a braid group action on $\calU_q(\g)$.
We briefly recall this action following \cite{LusztigBook}.
For each $i \in I$, we set $\sfS_i \seteq T_{i,-1}'$ and $\sfS_i^{*} \seteq T_{i,1}''$,
where $T_{i,-1}'$ and $T_{i,1}''$ denote Lusztig’s braid symmetries defined in \cite[Chapter~37]{LusztigBook} and described as follows:
\begin{subequations} \label{eq: Lusztig action}
  \begin{gather}
   \ba{lll}
&\sfS_i(t_i)\seteq t_i^{-1},&\quad  \sfS_i(t_j)\seteq t_jt_i^{-\sfc_{i,j}},\\
&\sfS_i(f_i) \seteq -e_it_i,&\quad
\sfS_i(f_j) \seteq \sum_{r+s = -\sfc_{i,j}} (-q_i)^s f_i^{(r)} f_j f_i^{(s)} \ \ ( i \ne j),\\
&\sfS_i(e_i) \seteq -t_i^{-1}f_i,&\quad
\sfS_i(e_j)\seteq \sum_{r+s = -\sfc_{i,j}} (-q_i)^{-r} e_i^{(r)} e_j e_i^{(s)}  \ \ ( i \ne j),
\ea\\[1ex]
\ba{lll}
&\sfS_i^*(t_i)\seteq t_i^{-1},&\quad  \sfS_i^*(t_j)\seteq t_jt_i^{-\sfc_{i,j}},\\
&\sfS_i^*(f_i) \seteq -t_i^{-1}e_i, &\quad \sfS_i^{*}(f_j)\seteq \sum_{r+s = -\sfc_{i,j}} (-q_i)^r f_i^{(r)} f_j f_i^{(s)}  \ \ ( i \ne j),\\
&\sfS_i^*(e_i) \seteq -f_it_i,&\quad
\sfS_i^*(e_j) \seteq \sum_{r+s = -\sfc_{i,j}} (-q_i)^{-s} e_i^{(r)} e_j e_i^{(s)}  \ \ ( i \ne j).
\ea
\end{gather}
\end{subequations} 
Here $f_i^{(n)}= f_i^n/[n]_i!$ and $e_i^{(n)}= e_i^n/[n]_i!$ for $n \in \Z_{\ge 1}$. 
Then we have
$\sfS_i^* \circ \sfS_i = \sfS_i \circ \sfS_i^* = {\rm id}$
(see also \cite{Saito94}) and 
the automorphisms $\{ \sfS_i\}_{i\in I}$ satisfy the relations of $\ttB_\sfC$ and hence $\ttB_\sfC$ acts on $U_q(\g)$ via $\{ \sfS_i\}_{i \in I}$. 

The braid group action on the bosonic extension $\hcalA$ is introduced in \cite{KKOP21B,JLO2,KKOP24B}.

\begin{theorem}[{\cite[Theorem 3.1]{KKOP24B}}]
For each $i \in I$, there exist unique $\Q(q)^{1/2}$-algebra automorphisms $\TT_i$  and $\TT_i^\star$ on $\hcalA$ such that
\begin{subequations}
\begin{align}
\TT_i(f_{j,m}) = \bc f_{j,m+1} & \text{ if $i =j$}, \\[1ex]
 \displaystyle\sum_{r+s = -\sfc_{i,j}}  (-q_i)^{s} \calF_{i,m}^{(r)} f_{j,m} \calF_{i,m}^{(s)}    & \text{ if $i \ne j$},
\ec
\end{align}
and
\begin{align}
\TT^\star_i(f_{j,m}) = \bc f_{j,m-1} & \text{ if $i =j$}, \\[1ex]
 \displaystyle\sum_{r+s = -\sfc_{i,j}}  (-q_i)^{r} \calF_{i,m}^{(r)} f_{j,m} \calF_{i,m}^{(s)} & \text{ if $i \ne j$},
\ec
\end{align}
\end{subequations}
where $\calF_{i,m} \seteq q_i^{1/2} (1-q_i^2)^{-1}f_{i,m}$ and $\calF_{i,m}^{(n)} \seteq \calF_{i,m}^n/[n]_i!$
for $n \in \Z_{\ge0}$. Moreover, we have 
\bnum
\item $\TT_i \circ \TT_i^\star = \TT_i^\star \circ \TT_i = \ {\rm id}$,
\item $\{ \TT_i \}_{i \in I}$ satisfy the relations of $\ttB$.     
  \ee
\end{theorem}

From the above theorem, for each $\ttb \in \ttB$ with $\ttb= \sigma_{i_1}^{\ep_1}\sigma_{i_2}^{\ep_2} \cdots
\sigma_{i_r}^{\ep_r}$ $(\ep_k \in \{ \pm 1 \})$, 
$$ \TT_\ttb \seteq \TT^{\ep_1}_{i_1}\TT^{\ep_2}_{i_2} \cdots \TT^{\ep_r}_{i_r} \text{ is well-defined}.$$
In particular $\TT_i = \TT_{\sigma_i}$. Note that, for any homogeneous element $x$, we have $\wt(\TT_i(x)) = s_i \wt(x)$.
Since $\TT_i^\star = \star \circ \TT_i \circ \star$, we have
\eq
\star\circ\TT_\ttb\circ\star=\TT_{\psi(\ttb)}\qt{for any $\ttb\in\ttB$.}
\label{eq:TTsta}
\eneq

\begin{lemma} [{\cite[Corollary 8.4 (b)]{JLO2}, \cite[Lemma 4.4]{OP24}}] \label{lem: shift}
For any $(i,p) \in I \times \Z$, we have
$$
\TT_\Updelta(f_{i,p}) = f_{i^*,p+1}. 
$$
\end{lemma} 

By Lemma~\ref{lem: shift}, we have
\eq \TT_{\Updelta^m} \hcalA_{<0} = \hcalA_{<m} \qtq
\TT_{\Updelta^m} \hcalA_{\ge0} = \hcalA_{\ge m}
\qt{for any $m\in\Z$}.\label{eq:Delta}
\eneq

\Lemma[{\cite[Proposition~4.7]{KKOP24B}}]\label{lem:int}
Let $\ttb=\sigma_{i_1}\cdots \sigma_{i_r}\in\ttB^+$,
and set $p_k=\TT_{i_1}\cdots\TT_{i_{k-1}}f_{i_k,0}$
for $1\le k\le r$.
Then, we have
\bnum
\item $\TT_\ttb\hcalA_{\ge0}\subset\hcalA_{\ge0}$ and $\TT_\ttb\hcalA_{<0}\supset\hcalA_{<0}$,
\item in particular 
  $\TT_{\ttb_1}\hcalA_{\ge0}\supset\TT_{\ttb_2}\hcalA_{\ge0}$ and $\TT_{\ttb_1}\hcalA_{<0}\subset\TT_{\ttb_2}\hcalA_{<0}$
  for any $\ttb_1,\ttb_2\in\ttB$ such that
  $\ttb_1\ledot \ttb_2$,
\item
$\hcalA(\ttb)\seteq\TT_\ttb \hcalA_{<0}\cap\hcalA_{\ge0} =\Q(q^{1/2})[p_r]
\tens\Q(q^{1/2})[p_{r-1}]\tens\cdots\tens\Q(q^{1/2})[p_1]$ as a $\Q(q^{1/2})$-vector space.
\ee
\enlemma

\section{Faithfulness} In this section, we prove the following theorem, which is the goal of this paper.

\begin{theorem} \label{thm: main}
The braid group action on $\hcalA$ via $\{ \TT_i\}_{i \in I}$ is faithful. 
\end{theorem}

Recall the endomorphisms $\rmE_{i,m}$ and $\Es_{i,m}$ in~\eqref{eq: E endo}, and the bilinear form $\hAform{ \ , \ }$ in~\eqref{eq: A form}.  

\begin{lemma} \label{lem: the lemma}
  Let $\ttb,\ttb_1,\ttb_2 \in \ttB^+$ and  $m\in\Z_{\ge0}$.
\bna
\item \label{it: (a)} If $i\in I$
  satisfies $\sigma_i \ttb \ledot \Updelta^{m+1}$, then we have $\rmE_{i,m}\TT_\ttb \hcalA_{<0} =0$.  
\item \label{it: (b)} If $\ttb \ledot \Updelta^{m}$ and $\sigma_i \not\ledot \ttb$, then $\sigma_i\ttb \ledot \Updelta^m$. 
\item \label{it: (c)} $\{ x \in \hcalA_{\le m} \ | \  \hAform{x,\sum_{i\in I} f_{i,m}\hcalA_{\le m}} =  0\} = \hcalA_{\le m-1}$. 
\item \label{it: (d)} If $m\ge 1$,  $\ttb_1,\ttb_2 \ledot \Updelta^m$ and $\ttb_1 \wedge \ttb_2 =1$, then $\TT_{\ttb_1}\hcalA_{<0} \cap \TT_{\ttb_2}\hcalA_{<0}  \subset \hcalA_{< m-1}$. 
\ee
\end{lemma}

\begin{proof}
\eqref{it: (a)} Note that $\TT_{\sigma_i \ttb} \hcalA_{<0} \subset \hcalA_{\le m}$ and $\TT_i f_{i,m+1} = f_{i,m+2} $.
Hence \eqref{eq:br} implies
$$
0 = [ \TT_{\sigma_i \ttb} \hcalA_{<0},  \TT_i f_{i,m+1} ]_q =\TT_i  [ \TT_{\ttb} \hcalA_{<0},   f_{i,m+1} ]_q = 
\TT_i(\rmE_{i,m}(\TT_{\ttb} \hcalA_{<0})),
$$
which implies the first assertion \eqref{it: (a)}. 

\noindent
\eqref{it: (b)} By applying
the anti-automorphism of $\ttB^+$ sending $\sigma_i$ to itself, we
can reduce the problem: if $\ttb \ledot \Updelta^m$ and $\ttb \not\in \ttB^+ \sigma_i$, then $\ttb \sigma_i \ledot \Updelta^m$. Write $\ttb = \ttb_1\ttb_2$ such that $\ttb_1 \seteq \Updelta^{m-1} \wedge \ttb$ and $\ttb_2 \ledot \Updelta$. Then we have $\ttb_2 \not\in \ttB^+ \sigma_i$ and hence $\ttb_2 \sigma_i \ledot \Updelta$. 

\snoi
\eqref{it: (c)} Let us set $S \seteq \sum_{i \in I} f_{i,m} \hcalA_{\le m} =\sum_{i \in I}  f_{i,m} \hcalA[m] \otimes \hcalA_{\le m-1}$. Then we have
$\hcalA_{\le m} = S \soplus \hcalA_{\le m-1}$ and 
$
\hAform{S,\hcalA_{\le m-1}}=0 
$.
Then the assertion follows from the fact that $\hAform{ \ , \ }$ on $\hcalA_{\le m}$ is non-degenerate. 

\snoi
\eqref{it: (d)}
By Lemma~\ref{lem:int} and \eqref{eq:Delta},
we have $\TT_{\ttb_k}(\hcalA_{<0})\subset\hcalA_{< m}$
for $k=1,2$.   By the assumption, for any $i \in I$, we have $\sigma_i \not\ledot \ttb_1$ or $\sigma_i \not\ledot \ttb_2$. Thus for each $i \in I$,
there exists $k_i \in \{1,2\}$  such that $\sigma_i \ttb_{k_i} \in \Updelta^m$ by~\eqref{it: (b)}. Then $\rmE_{i,m-1}( \TT_{\ttb_1}\hcalA_{<0} \cap \TT_{\ttb_2}\hcalA_{<0}) =0$ by~\eqref{it: (a)}. Then  the assertion follows from~\eqref{it: (c)} and Theorem~\ref{thm: hAform}~\eqref{it: pairing}.
\end{proof}

\begin{proposition} \label{prop: the prop}
For $\ttb_1,\ttb_2 \in \ttB$ and $m\in\Z$, we have
$$
\TT_{\ttb_1}\hcalA_{<m} \cap \TT_{\ttb_2}\hcalA_{<m} = \TT_{\ttb_1 \wedge \ttb_2}\hcalA_{<m}
\qtq
\TT_{\ttb_1}\hcalA_{ \ge m} \cap \TT_{\ttb_2}\hcalA_{\ge m} = \TT_{\ttb_1 \vee \ttb_2}\hcalA_{\ge m}.
$$
\end{proposition}

\begin{proof}
  By Theorem~\ref{thm: Garside}, \eqref{eq:psibb} and \eqref{eq:TTsta},
  it is enough to show that
  $$\TT_{\ttb_1}\hcalA_{<0} \cap \TT_{\ttb_2}\hcalA_{<0} = \TT_{\ttb_1 \wedge \ttb_2}\hcalA_{<0}\quad\text{for any $\ttb_1,\ttb_2 \in \ttB^+$.}
  $$
Let us show it by induction on $m \ge 0$    
such that $\ttb_1,\ttb_2 \ledot \Updelta^m$. The case $m=0$ is trivial.
Thus let us assume further that $m \ge 1$.

We first claim that 
\begin{align}\label{eq: claim}
\TT_\ttb \hcalA_{<0}  \cap \hcalA_{<m-1}     = \TT_{\Updelta^{m-1} \wedge \ttb } \hcalA_{<0}\qt{for any $m\in\Z_{\ge0}$ and $\ttb\in\ttB^+$ such that $\ttb \ledot \Updelta^m$.}
\end{align}
Since \eqref{eq: claim} is trivial for $m\le1$, we may assume that $m\ge2$.
Let us write $\ttb = \ttb_{(1)}\ttb_{(2)}$ such that $1\ledot\ttb_{(1)} = \Updelta^{m-1} \wedge \ttb$ and $1\ledot\ttb_{(2)} \ledot \Updelta$. 
Then we have
\begin{align*}
\TT_{\ttb} \hcalA_{<0}   \cap \hcalA_{<m-1} & = \TT_{\ttb} \hcalA_{<0} \cap \TT_{\Updelta^{m-1}} \hcalA_{<0}   \\
& = \TT_{\ttb_{(1)}}(\TT_{\ttb_{(2)}} \hcalA_{<0} \cap \TT_{\ttb_{(1)}^{-1} \Updelta^{m-1}} \hcalA_{<0} ) \\
& \underset{*}{=} \TT_{\ttb_{(1)}}(\TT_{\ttb_{(2)} \wedge \ttb_{(1)}^{-1} \Updelta^{m-1}} \hcalA_{<0} ) \\
& = \TT_{\ttb \wedge \Updelta^{m-1}} \hcalA_{<0}.
\end{align*}
Here $\underset{*}{=} $ holds by the induction hypothesis, since $\ttb_{(2)},\ttb_{(1)}^{-1}\Updelta^{m-1} \ledot \Updelta^{m-1}$.
Hence~\eqref{eq: claim} holds. 

\smallskip
Now let us write $\ttb_k = \ttx \tty_k$ $(k=1,2)$ such that $\ttx = \ttb_1 \wedge \ttb_2$ and $\tty_1 \wedge \tty_2 =1$. Note that $\tty_k \ledot \Updelta^m$. Then we have
$$
\TT_{\ttb_1}\hcalA_{<0} \cap \TT_{\ttb_2}\hcalA_{<0} = \TT_\ttx (\TT_{\tty_1}\hcalA_{<0} \cap \TT_{\tty_2}\hcalA_{<0}).
$$

  Then,  we have 
\begin{align*}
\TT_{\tty_1}\hcalA_{<0} \cap \TT_{\tty_2}\hcalA_{<0} & = \TT_{\tty_1}\hcalA_{<0} \cap \TT_{\tty_2}\hcalA_{<0} \cap \hcalA_{<m-1} 
\ \ \text{by Lemma~\ref{lem: the lemma}~\eqref{it: (d)} }\\
& = (\TT_{\tty_1}\hcalA_{<0} \cap \hcalA_{<m-1} )\cap (\TT_{\tty_2}\hcalA_{<0} \cap \hcalA_{<m-1}) \\
& =  \TT_{\Updelta^{m-1} \wedge \tty_1}\hcalA_{<0} \cap \TT_{\Updelta^{m-1} \wedge\tty_2}\hcalA_{<0}  \ \ \text{by \eqref{eq: claim}}  \\
& =  \TT_{\Updelta^{m-1} \wedge \tty_1 \wedge \Updelta^{m-1} \wedge\tty_2}\hcalA_{<0} = \hcalA_{<0} \ \ \ \text{by the induction on $m$}.
\end{align*}
Hence 
\[
\TT_{\ttb_1}\hcalA_{<0} \cap \TT_{\ttb_2}\hcalA_{<0} = \TT_{\ttx}\hcalA_{<0}. \qedhere
\]
\end{proof}

\begin{proof} [Proof of Theorem~\ref{thm: main}]

  It is enough to show that, if $\ttb\in\ttB$ satisfies
  $\TT_\ttb \hcalA_{<0} = \hcalA_{<0}$ then $\ttb=1$.

  \snoi
  (a)\ If $\ttb\in\ttB^+$ satisfies $\TT_\ttb \hcalA_{<0} = \hcalA_{<0}$,
  then $\TT_\ttb \hcalA_{<0}\cap\hcalA_{\ge0}=\Q(q^{1/2})$, and hence
  Lemma~\ref{lem:int} implies $\ttb=1$.
  Hence, if $\ttb\in\ttB^-$ satisfies $\TT_\ttb \hcalA_{<0} = \hcalA_{<0}$
  then $\ttb=1$.

  \snoi
  (b)
  If $\TT_\ttb \hcalA_{<0} = \hcalA_{<0}$, then one has
  \eqn
\hcalA_{<0}=\TT_\ttb \hcalA_{<0} \cap \hcalA_{<0} = \TT_{\ttb \wedge 1} \hcalA_{<0} 
\eneqn
by Proposition \ref{prop: the prop}.
Since  $\ttb\wedge 1 \in \ttB^-$, $\ttb\wedge1=1$ by (a).
Hence $\ttb\in\ttB^+$.
Then (a) implies $\ttb=1$.
\end{proof}


\providecommand{\bysame}{\leavevmode\hbox to3em{\hrulefill}\thinspace}
\providecommand{\MR}{\relax\ifhmode\unskip\space\fi MR }
\providecommand{\MRhref}[2]{%
  \href{http://www.ams.org/mathscinet-getitem?mr=#1}{#2}
}
\providecommand{\href}[2]{#2}

\end{document}